\documentclass{article}

\usepackage{amsmath,amssymb,amsthm,bm}
\usepackage{bbm}
\usepackage{enumerate}
\usepackage{setspace}
\usepackage{tikz-cd} 
\usepackage[utf8]{inputenc}
\usepackage{xcolor}
\usepackage{tikz}
\usetikzlibrary[shadings]
\usepackage{url}

\makeatletter
\providecommand*{\diff}%
{\@ifnextchar^{\DIfF}{\DIfF^{}}}
\def\DIfF^#1{%
	\mathop{\mathrm{\mathstrut d}}%
	\nolimits^{#1}\gobblespace}
\def\gobblespace{%
	\futurelet\diffarg\opspace}
\def\opspace{%
	\let\DiffSpace\!
	\ifx\diffarg(%
	\let\DiffSpace\relax
	\else
	\ifx\diffarg[%
	\let\DiffSpace\relax
	\else
	\ifx\diffarg\{%
	\let\DiffSpace\relax
	\fi\fi\fi\DiffSpace}
\makeatother

\newtheorem{theorem}{Theorem}
\newtheorem*{theorem*}{Theorem}

\newtheorem*{corollary*}{Corollary}
\newtheorem{lemma}{Lemma}
\newtheorem*{proposition*}{Proposition}
\newtheorem{proposition}{Proposition}

\theoremstyle{definition}

\theoremstyle{remark}

\newtheorem*{remark}{Remark}

\newcommand{\ord}{k}

\title{A sharp higher order Sobolev embedding}
\author{Raul Hindov, Shahaf Nitzan\footnote{The second author is supported by NSF CAREER grant DMS 1847796.}, Jan-Fredrik Olsen, Eskil Rydhe\footnote{The fourth author is supported by VR grant 2022-04307.}}

\begin{document}

	\maketitle
	
	\begin{abstract}
		\textcolor{black}{We obtain sharp embeddings from the Sobolev space $W^{\ord,2}_0(-1,1)$ into the space $L^1(-1,1)$ and determine the extremal functions.  This improves on a previous estimate of the sharp constants of these embeddings due to Kalyabin.}
	\end{abstract}
	
	\section{Introduction}
	In this paper we give a proof of a Sobolev-type inequality with a sharp constant and an explicit extremal function. This is motivated by the more general problem of calculating sharp constants and identifying extremal functions for Sobolev embeddings $W^{\ord,p}_0(-1,1) \subset L^q(-1,1)$. That is,   inequalities of the form
	\begin{equation*}
	\left(\int_{-1}^1|f|^q \diff x\right)^{\frac{1}{q}}\leq c_{\ord,p,q} \left(\int_{-1}^1\left| f^{(\ord)}\right|^p\diff x\right)^{\frac{1}{p}}
	\end{equation*}
	for functions $f: \mathbb{R} \rightarrow \mathbb{R}$ with support in $[-1,1]$ and that satisfy $f^{(k)} \in L^p(\mathbb{R})$. In particular,  for integers $\ord \geq 1$, $f$ satisfies the boundary conditions $f^{(j)}(\pm 1)=0$  for all $0 \le j < \ord$.
	
	We shall consider the case $q=1$, $p=2$, and integers $k \geq 1$, for which the sharp constants and extremal functions do not seem to be known. See, e.g., the surveys by Mitrinovi{\'c} et al. \cite[Chapter II]{Mitrinovic},  Kuznetsov and Nazarov \cite{Kuznetsov}, and Nazarov and Shcheglova \cite{Nazarov_Shcheglova2021}.
	
	Our main result is as follows.
	\begin{theorem}\label{theorem}
		For all integers $k \geq 1$ and $f\in W^{\ord,2}_0(-1,1)$,  we have the sharp inequality
		\begin{equation}\label{sob_ineq}
			\int_{-1}^1|f(x)| \diff x\leq \frac{1}{(2\ord-1)!!\sqrt{\ord+\frac12}}\left(\int_{-1}^1|f^{(\ord)}(x)|^2 \diff x\right)^{\frac{1}{2}}.
		\end{equation}
		The extremal functions are given by the Landau kernels, $L_\ord(x)=(1-x^2)^\ord$.
	\end{theorem}
	
	This theorem improves a bound due to Kalyabin \cite{Kalyabin2}. Indeed, for $p=2$, $q \in (0,\infty)$ and $k \in \mathbb{Z}_+$, Kalyabin obtained that
	\begin{equation*}
	\frac{\sqrt{\ord+\frac12}}{2^\ord \ord!}\left(\mathcal{L}(\ord q)\right)^{1/q}
	\leq
	c_{\ord,2,q}
	\leq
	\frac{\left(\mathcal{L}((\ord-\frac12)q)\right)^{1/q}}{2^\ord(\ord-1)!\sqrt{\ord-\frac12}},
	\end{equation*}
	where $\mathcal{L}(s)=\int_{-1}^{1}(1-x^2)^s \diff x$.  A straight-forward calculation reveals that the constant from Theorem \ref{theorem} is identical to Kalyabin's lower estimate in the case $q=1$, implying that it was indeed sharp.

	It is well-known that sharp constants of Sobolev embeddings can be connected to minimal eigenvalues of certain boundary value problems (see, e.g.,  \cite{Widom}). In our case, the sharp constant in Theorem \ref{theorem} is connected to the minimal eigenvalue of the boundary value problem
	\begin{equation*}
	\left\{
	\begin{aligned}
	(-1)^\ord u^{(2\ord)}(x) &=\lambda \frac{u(x)}{|u(x)|}\int_{-1}^{1}|u(t)|\diff t,  &&x\in [-1,1], \\[2mm]
	u^{(j)}(\pm 1)&=0,  &&j \in\{0,1,\ldots, 2\ord-1\}.
	\end{aligned}
	\right.
	\end{equation*}
	Indeed, we have
	\begin{align*}
	\min_{u\ne 0}\frac{\left((-1)^ku^{(2k)},u\right)}{\left(\frac{u}{|u|}\left\Vert u\right\Vert_{L^1(-1,1)},u\right)}&= \min_{u\ne 0}\frac{\left(u^{(k)},u^{(k)}\right)}{\left\Vert u\right\Vert^2_{L^1(-1,1)}}\\[2mm]
	&= \min_{u\ne 0}\frac{\left\Vert u^{(\ord)} \right\Vert_{L^2(-1,1)}^2}{\left\Vert u\right\Vert^2_{L^1(-1,1)}}=\left(\frac{1}{c_{\ord,2,1}}\right)^2,
	\end{align*}
	where $(\cdot,\cdot)$ is the inner product in $L^2(-1,1)$ and the first equality follows upon $k$ times partially integrating and using the boundary conditions. In  \cite{Widom}, a similar example is provided for the case $p=q=2$.
	
\section{Proof of Theorem \ref{theorem}}
	The main idea of the proof is to consider a class of explicit left inverses of the differential operator $f\mapsto f^{(\ord)}$ of the form
	\begin{align*}
		f(x)=\int_{-1}^1 B_\ord(x,y)f^{(\ord)}(y)\diff y,\quad x\in\mathbb{R}.
	\end{align*}
	By the Cauchy-Schwarz inequality, we obtain from this integral representation that
	\begin{align*}
		\int_{-1}^1 |f(x)|\diff x \leq \left\Vert f^{(\ord)}\right\Vert_{L^2(-1,1)}
		\left\Vert \int_{1}^1|B_\ord(x,y)|\diff x \right\Vert_{L^2_{(-1,1;\diff y)}}.
	\end{align*}
	The inequality of Theorem \ref{theorem} is then obtained by connecting the expression involving the norm of $B_\ord(x,y)$ to a minimization problem having Legendre polynomials as minimizers. Finally, sharpness is obtained by noting that Landau kernels are extremal functions for the resulting inequality.
	
	Before proceeding with a proof, we discuss some notation. For  integers $\ord \geq 1$ and $p\in [1,\infty)$
	we define the Sobolev space
	\begin{equation*}
	W^{\ord,p}_0(-1,1)=\left\{ f:\mathbb{R}\to \mathbb{R} \mid \operatorname{supp} f \subseteq[-1,1],\ f^{(\ord)}\in L^p(\mathbb{R}) \right\}.
	\end{equation*}
	For each integer $0\le j < \ord$, the derivative $f^{(j)}$ is absolutely continuous since $f^{(j-1)}(x)=\int_{-1}^{x}f^{(j)}(t)\diff t$ and $L^p(-1,1)\subset L^1(-1,1)$. Hence, as mentioned in the introduction, since $f$ has support in $[-1,1]$, it follows that $f^{(j)}(\pm 1)=0$ for each such $j$.
	The norm of $f$ in $W^{\ord,p}_0(-1,1)$ is defined by $\Vert f \Vert_{W^{\ord,p}_0(-1,1)} = \Vert f^{(k)} \Vert_{L^p(-1,1)}$.  
		By $\mathbbm{1}_X$, we denote the function that equals 1 if condition $X$ is satisfied, and 0 otherwise. By, $\delta_{nm}$ we denote the standard Kronecker delta function.
		
	Now we give the details of the proof.
	
	\subsection{Construction of the explicit left inverses}
	
	For integers $\ord \geq 1$, define the functions $b_\ord\colon \mathbb{R}^2\to \mathbb{R}$ by
	\begin{align*}
		b_\ord(x,y)=\frac{(x-y)^{\ord-1}}{(\ord-1)!}\left[\mathbbm{1}_{y<x<0}-\mathbbm{1}_{y>x\geq 0}\right].
	\end{align*}
	The support of $b_\ord$ is indicated in Figure~\ref{fig:SupportOfBp1}.
	\begin{figure}
		\centering
		\begin{tikzpicture}[scale=1]
			\draw[thin,->] (0,2.5) -- (8,2.5) node[right] {$x$};
			\draw[thin,->] (4,0) -- (4,5) node[right] {$y$};
			
			\draw[dashed, very thin] (0,0) -- (8,5);
			\shade [left color=orange, middle color=white, shading angle=30] (4,2.55)--(4,4.7)--(7.5,4.7);
			\shade [right color=teal, middle color=white, shading angle=30] (3.95,2.45)--(3.95,0.3)--(.5,0.3);
			\draw (1.6,3) node {$0$};
			\draw (2.3,1.6) node[anchor = north west] {$\frac{(x-y)^{\ord-1}}{(\ord-1)!}$};
			\draw (6.4,2) node {$0$};
			\draw (6,3.4) node[anchor = south east] {$-\frac{(x-y)^{\ord-1}}{(\ord-1)!}$};

		\end{tikzpicture}
		\caption{Illustration of the support and values of $b_\ord(x,y)$. The thin dashed line indicates $y=x$. $b_\ord(x,y)$ is equal to zero outside of the shaded area between the $y$-axis and the line $y=x$.}\label{fig:SupportOfBp1}
	\end{figure}
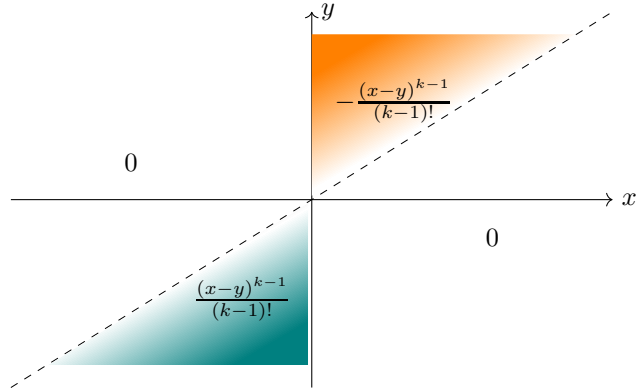
	
	\begin{proposition}\label{proposition:PropertiesOfBp1}\phantom{.}
		For all  integers $\ord \geq 1$, the following holds.
		\begin{enumerate}[$(i)$]
			\item\label{vertical} $\displaystyle \int_{-\infty}^\infty b_\ord(x,y)\diff x = \frac{(-y)^\ord}{\ord!}$ for all $y\in\mathbb{R}$.
			\item\label{limits} The following one-sided limits hold.
			\begin{align*}
			\lim_{x\to 0^-} b_\ord(x,y)&=\left\{ \begin{aligned} 0 & \quad \text{for} \quad y> 0 \\[2mm]
				\frac{(-y)^{\ord-1}}{(\ord-1)!} & \quad \text{for} \quad y\leq0 \end{aligned} \right. \\[2mm]
				b_k(0,y) = \lim_{x\to 0^+} b_\ord(x,y)&=\left\{ \begin{aligned} \frac{-(-y)^{\ord-1}}{(\ord-1)!} & \quad \text{for} \quad y> 0 \\
					0 & \quad \text{for} \quad y\leq0 \end{aligned} \right. .
				\end{align*}
			\item\label{repr} Let $f\in W^{\ord,2}_0(-1,1)$ and $Q$ be a polynomial with $\deg Q<\ord$.Then, for all $x\in \mathbb{R}$,
			\begin{equation*}
				\int_{-\infty}^\infty \left[b_\ord(x,y)-Q(y)\right]f^{(\ord)}(y)\diff y = f(x).
			\end{equation*}
		\end{enumerate}
	\end{proposition}
	\begin{remark}
		By part (\ref{repr}) of the proposition, the integral operator induced by $b_\ord$ is a left-inverse to the differential operator $f\mapsto f^{(\ord)}$. Moreover, such a left-inverse is not unique; the polynomial $Q$, which may depend on $x$, provides a parametrization of a family of left-inverses. By part (\ref{limits}), $b_\ord(x,y)$ is discontinuous across the $y$-axis whenever $y\neq 0$. 
	\end{remark}
	\begin{proof}
		Property (\ref{vertical}) follows by direct computation, and property (\ref{limits}) follows immediately from the definition of $b_\ord(x,y)$.
		
		To prove the reproducing property (\ref{repr}), we use integration by parts repeatedly.
		Indeed,	for $x<0$, we have
		\begin{equation*}
			\int_{\mathbb{R}}b_\ord(x,y)f^{(\ord)}(y)\diff y 
				=\sum_{j=1}^{\ord -1}\left[\frac{(x-y)^{\ord-j}}{(\ord -j)!}f^{(\ord-j)}(y)\right]_{-1}^x+
				\int_{-1}^{x}f'(y)\diff y=f(x),
		\end{equation*}
		where all the terms in the sum vanish due to the boundary conditions on $f$.
		The case $x\geq 0$ is treated similarly.
		Finally, since $Q$ is a polynomial of $\deg Q<\ord$, it follows that
		\begin{equation*}
			\int_{\mathbb{R}}Q(y)f^{(\ord)}(y)\diff y = 0.
		\end{equation*}
	\end{proof}
	
	\begin{lemma}\label{lemma:PropertiesOfB}\phantom{.}
		Given any integer $\ord \ge 1$, let
		\begin{align*}
			f_\ord(y)=y^{\ord-1}\mathbbm{1}_{y>0},\quad y\in\mathbb{R}.
		\end{align*}
		Moreover, let $Q$ be a polynomial of degree at most $\ord - 1$, such that $f_\ord(y_n)=Q(y_n)$ for $\ord+1$ distinct real numbers $y_1 < \ldots < y_{\ord+1}$. \textcolor{black}{Then either  $Q(y) \equiv y^{k-1}$ or $Q(y) \equiv 0$. In particular,
		this implies that either $y_1 \geq 0$ or $y_{k+1} \leq 0$.}
	\end{lemma}
	\begin{proof}	
	If $\ord =1$, then $Q$ is a constant function. From this, it follows immediately that if $Q(y) -  \mathbbm{1}_{y>0}$ vanishes at two distinct points, then  either \textcolor{black}{$Q(y) \equiv 1$ or $Q(y) \equiv 0$}. Moreover, these points have to be both \textcolor{black}{non-positive} or both  \textcolor{black}{non-negative}. This argument is easily extended to $\ord = 2$ using the linearity of $Q$ and the hypothesis that $Q(y) - y\mathbbm{1}_{y>0}$ is to vanish at three distinct points.

	For general $k > 2$, we proceed by induction. To this end, suppose that $Q$ is a polynomial of degree at most $k-1$ and that 
		$g(y) = Q(y) - y^{k-1}\mathbbm{1}_{y>0}$ vanishes at $k+1$ distinct points. As $g$ has a continuous
		derivative, the mean value theorem implies that
		$g'(y) = Q(y) - (k-1)y^{k-2}\mathbbm{1}_{y>0}$ vanishes at at least $k$ distinct points. It follows by the induction hypothesis, that these $k$ points are either all   \textcolor{black}{non-positive} or all  \textcolor{black}{non-negative}. In particular, this means that $Q(y) - y^{k-1}\mathbbm{1}_{y>0}$ vanishes at $k$ points that are either all  \textcolor{black}{non-positive} or all  \textcolor{black}{non-negative}. Since $Q$ is of degree $k-1$, the conclusion follows.
	\end{proof}
	\begin{remark}
	For $k\geq 1$, let $f_k(y)$ be as above. Then the lemma implies that if a polynomial $Q$ of degree at most $k-1$ is equal to ${f}_k(x^\ast - y)$ at $k+1$ distinct points $y_1 < \cdots < y_{k+1}$, then either $y_1 \geq x^\ast$ or $y_{k+1} \leq x^\ast$.
	\end{remark}
	
	We now fix $\ord$ distinct real numbers $y_1<y_2<\cdots < y_\ord$, and consider the corresponding Lagrange interpolation basis $\{p_n\}_{n=1}^\ord$ given by 
	\begin{equation*}
	p_n(y)=\prod_{\substack{1\le j\le \ord; \\ j \ne n}}\frac{y-y_j}{y_n-y_j}.
	\end{equation*}
	These polynomials are of degree $\ord-1$ and satisfy $p_n(y_m)=\delta_{nm}$. 
	Using these polynomials, we define functions $B_\ord\colon\mathbb{R}^2\to\mathbb{R}$ by
	\begin{align*}
		B_\ord(x,y)
		=
		b_\ord(x,y) - \sum_{n=1}^\ord p_n(y) b_\ord(x,y_n).
	\end{align*}
	\begin{proposition}\label{proposition:PropertiesOfB2}\phantom{.}
		For all integers $\ord \geq 1$ the following holds.
		\begin{enumerate}[$(i)$]
			\item\label{horiz_zeros}  For all $n\in\{1,2,\ldots,\ord\}$ and  $x \in \mathbb{R}$, we have $B_\ord(x,y_n)=0$.
			\item\label{repr_prop}  For all $f\in W_0^{\ord,2}(-1,1)$ and  $x\in\mathbb{R}$, we have $\int_{-\infty}^\infty \hspace{-0.5mm} B_\ord(x,y) \hspace{-0.5mm} f^{(\ord)}(y)\diff y \hspace{-0.5mm} = \hspace{-0.5mm} f(x)$.
			\item\label{cont_prop}  For all $y\in\mathbb{R}$, the function $x\mapsto B_\ord(x,y)$ is continuous. 
			\item\label{out_of_y1_yk} For all $y \in \mathbb{R}$, 
			\begin{equation*}
				B_k(x,y) = 
				\left\{ 
				\begin{aligned}
					\frac{(x-y)^{\ord-1}}{(\ord-1)!}\mathbbm{1}_{y<x}, && x \leq y_1 \\[2mm]
					-\frac{(x-y)^{\ord-1}}{(\ord-1)!}\mathbbm{1}_{y>x}, && x \geq y_k
				\end{aligned}				
				\right..
			\end{equation*}
			\item\label{constant_sign}  For all $y\in\mathbb{R}$, either $B_\ord(x,y)\ge 0$ or $B_\ord(x,y)\le 0$ for all $x\in\mathbb{R}$.
		\end{enumerate}
	\end{proposition}
	\begin{proof}
		Property (\ref{horiz_zeros}) is immediate from $p_n(y_m)=\delta_{nm}$. Note that for fixed $x$, $B_\ord(x,y)=b_\ord(x,y)-Q(y)$ for a polynomial $Q$ of degree at most $\ord -1$. Hence, (\ref{repr_prop}) follows from   Proposition~\ref{proposition:PropertiesOfBp1}~(\ref{repr}).
		
		To establish (\ref{cont_prop}), we note that by Proposition~\ref{proposition:PropertiesOfBp1}~(\ref{limits}),
		for any fixed $y\in\mathbb{R}$, $x\mapsto b_\ord(x,y)$ is continuous apart from a jump discontinuity at $x=0$. Consequently, $x\mapsto B_\ord(x,y)$ can only fail to be continuous at $x=0$. We therefore compare $\lim_{x\to 0^-}B_\ord(x,y)$ with $B_\ord(0,y) = \lim_{x\to 0^+}B_\ord(x,y)$. Assuming first that $y>0$, we obtain from Proposition~\ref{proposition:PropertiesOfBp1}~(\ref{limits}) that
		\begin{align*}
			\lim_{x\to 0^-}B_\ord(x,y)&= - \sum_{n=1}^\ord p_n(y) \frac{(-y_n)^{\ord-1}}{(\ord-1)!}\mathbbm{1}_{y_n<0},\\
			\lim_{x\to 0^+}B_\ord(x,y)&= -\frac{(-y)^{\ord-1}}{(\ord-1)!} + \sum_{n=1}^\ord p_n(y) \frac{(-y_n)^{\ord-1}}{(\ord-1)!}\mathbbm{1}_{y_n>0}.
		\end{align*}
		By similar calculations for $y\leq 0$, it therefore holds for all $y\in\mathbb{R}$ that
		\begin{align*}
			\lim_{x\to 0^-}B_\ord(x,y) - \lim_{x\to 0^+}B_\ord(x,y)= \frac{(-y)^{\ord-1}}{(\ord-1)!} - \sum_{n=1}^\ord p_n(y) \frac{(-y_n)^{\ord-1}}{(\ord-1)!}.
		\end{align*}
		As a function of $y$, the above right-hand side vanishes at each of the points $y_1, y_2, \ldots, y_k$, and since it is a polynomial of degree at most $\ord-1$ it must be identically equal to $0$.
		
		To establish (\ref{out_of_y1_yk}) for the case $x\le y_1$, we note that the desired conclusion is immediate from the definition of $B_k$ if $x<0$. For the remaining case we fix $x$ so that $0\le x\le y_1$ and consider the corresponding expression
		\begin{align*}
			B_\ord(x,y)=-\frac{(x-y)^{\ord-1}}{(\ord-1)!}\mathbbm{1}_{y>x}+\sum_{n=1}^\ord p_n(y) \frac{(x-y_n)^{\ord-1}}{(\ord-1)!}.
		\end{align*}
		As a function of $y$, the sum in the above right-hand side is a polynomial of degree at most $\ord-1$. Moreover, it  is equal to the polynomial $\frac{(x-y)^{\ord-1}}{(\ord-1)!}$ at each $y_n$, and so these polynomials must coincide. This establishes the case  $x\le y_1$. For the case $x\ge y_\ord$, the result follws by analogous arguments.

		\textcolor{black}{To establish (\ref{constant_sign}), we first note that by (i) and (iii), the function $x \mapsto B_k(x,y)$ is continuous  for all $y \in \mathbb{R}$ and identically equal to zero for  each $y \in \{y_1, \ldots, y_k\}$. Moreover, by (iv), $B_k(x,y)$ does not change sign  on the sets  $\{(x,y) : x \leq y_1\}$ and  $\{(x,y) : x \geq y_k\}$, respectively  (cf., Figure~\ref{fig:SupportOfB3}). Therefore,  if $x \rightarrow B_k(x,y)$ changes sign for any fixed $y^\ast \notin \{y_1, \ldots, y_k\}$,   there must exist a point $x^\ast \in (y_1,y_k)$ so that $B_k(x^\ast, y^\ast)=0$. In particular, this means that  $y \mapsto B_k(x^\ast,y)$ vanishes at the $\ord + 1$ distinct points $\{y_1, \ldots, y_\ord, y^\ast\}$.
		To  show that this leads to a contradiction, 
		we  consider the cases  $x^*\ge 0$ and $x^*<0$  separately.  In the first case, it holds that
		\begin{align*}
			y \longmapsto B_\ord(x^*,y) = -\frac{(x^*-y)^{\ord-1}}{(\ord - 1)!}\mathbbm{1}_{y>x^*}  - Q(y),
		\end{align*}
		where $Q$ is some polynomial of degree at most $\ord - 1$.  Hence, by   Lemma~\ref{lemma:PropertiesOfB} and the remark following it, this implies that the  points $\{y_1, \ldots, y_\ord, y^\ast\}$ are either all smaller than $x^\ast$ or all greater than $x^*$. In particular, either $y_k \leq x^\ast$ or $y_1 \geq x^\ast$. This   contradicts the assumption that $x^*\in (y_1,y_\ord)$. The case $x^*<0$ may be treated similarly.}
	\end{proof}
	
	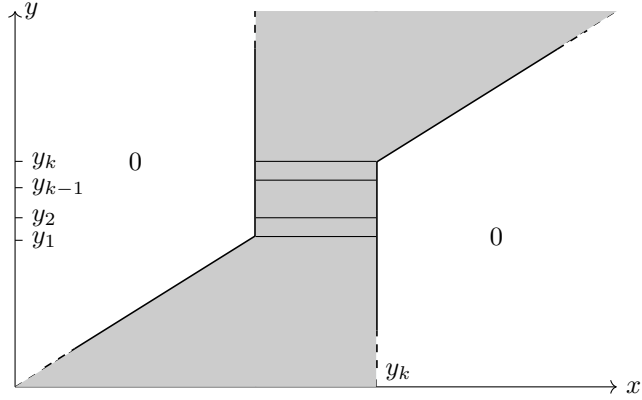
\begin{figure}
		\centering
		\begin{tikzpicture}[scale=1]
			\fill [teal!10] (0,0) -- (3.2,2) -- (3.2,5) -- (8,5) -- (4.8,3) -- (4.8,0);
			\draw[thin,->] (0,0) -- (8,0) node[right] {$x$};
			\draw[thin,->] (0,0) -- (0,5) node[right] {$y$};
			
			\draw[dashed, very thin] (0,0) -- (8,5);
			
			\draw[dashed, very thick] (0.4,0.25) -- (0.8,0.5);
			\draw[very thick] (0.8,0.5) -- (3.2,2);
			\draw[very thick] (4.8,3) -- (7.2,4.5);
			\draw[dashed, very thick] (7.2,4.5) -- (7.6,4.75);
			
			\draw[dashed, very thick] (3.2,0.25) -- (3.2,0.75);
			\draw[very thick] (3.2,0.75) -- (3.2,4.5);
			\draw[dashed, very thick] (3.2,4.5) -- (3.2,5);
			
			\draw[dashed, very thick] (4.8,0.25) -- (4.8,0.75);
			\draw[very thick] (4.8,0.75) -- (4.8,4.5);
			\draw[dashed, very thick] (4.8,4.5) -- (4.8,5);
			
			\draw (1.6,3) node {$0$};
			\draw (1.4,0.9) node[anchor = north west] {$\frac{(x-y)^{\ord-1}}{(\ord-1)!}$};
			\draw (6.4,2) node {$0$};
			\draw (6.8,4.1) node[anchor = south east] {$-\frac{(x-y)^{\ord-1}}{(\ord-1)!}$};
			
			\draw[thin] (0,1.95) -- (0.1,1.95) node[right] {$y_1$};
			\draw[thin] (0,2.25) -- (0.1,2.25) node[right] {$y_2$};
			\draw[thin] (0,2.65) -- (0.1,2.65) node[right] {$y_{\ord-1}$}; 
			\draw[thin] (0,3) -- (0.1,3) node[right] {$y_\ord$}; 
			\draw[thin] (3.2,0) -- (3.2,0.1);
			\draw[thin] (4.8,0) -- (4.8,0.1); 
			\draw[thin] (3.2,0.2) node[right] {$y_1$};
			\draw[thin] (4.8,0.2) node[right] {$y_\ord$}; 
			\fill [transparent,gray!40] (0,0) -- (3.2,2) -- (3.2,5) -- (8,5) -- (4.8,3) -- (4.8,0);
			\draw[thin] (3.2,3) -- (4.8,3);
			\draw[thin] (3.2,2.75) -- (4.8,2.75);
			\draw[thin] (3.2,2.25) -- (4.8,2.25);
			\draw[thin] (3.2,2) -- (4.8,2);
		\end{tikzpicture}
		\caption{Indication of the values of $B_{\ord}$ in various regions separated by thick lines, when $y_1>0$. The thin dashed line indicates $y=x$. The horisontal lines at $y \in \{y_1, \ldots, y_k\}$ indicate the zero set of $B_k$ for $x$ in the convex hull of $\{y_1,y_k,y\}$.}\label{fig:SupportOfB3}
	\end{figure}
	
	\subsection{Connection to a minimisation problem that leads to the Sobolev-type inequality (\ref{sob_ineq})}
	
	Suppose that $f\in W_0^{\ord,2}(-1,1)$  and $B_\ord(x,y)$ as above. Then we have
	\begin{align*}
		f(x)=\int_{-1}^1 B_\ord(x,y)f^{(\ord)}(y)\diff y,\quad x\in\mathbb{R}.
	\end{align*}
	Applying the  Cauchy--Schwarz inequality, we obtain that
	\begin{align*}
		\|f\|_{L^1(-1,1)}
		&\le 
		\left\Vert f^{(\ord)}\right\Vert_{L^2(-1,1)} \left\Vert\int_{-1}^1 |B_\ord(x,y)|\diff x\right\Vert_{L^2(-1,1;\diff y)}
		\\
		&=
		\left\Vert f^{(\ord)}\right\Vert_{L^2(-1,1)}\left\Vert\int_{-1}^1 B_\ord(x,y)\diff x\right\Vert_{L^2(-1,1;\diff y)},
	\end{align*}
	where the final equality is immediate from Proposition~\ref{proposition:PropertiesOfB2}(\ref{constant_sign}).
	
	It follows from Proposition~\ref{proposition:PropertiesOfBp1}(\ref{vertical}) that
	\begin{align*}
		\ord!\int_{-1}^1 B_\ord(x,y)\diff x
		=
		(-y)^\ord - \sum_{n=1}^\ord (-y_n)^{\ord}p_n(y).
	\end{align*}
	Since the polynomials $p_n$ are of degree $\ord-1$ and satisfy $p_n(y_m)=\delta_{nm}$, for $n,m\in \{1,2,\ldots,\ord\}$, the   right-hand side of the above expression is a monic polynomial of degree $\ord$  with distinct zeros $y_1, y_2,\ldots, y_\ord$.
	Since we can choose the zeros $y_1,y_2,\ldots,y_\ord$ freely,   any monic polynomial with distinct zeros can be obtained in this way. It follows that
	\begin{align*}
		\|f\|_{L^1(-1,1)}
		\le 
		\frac{\|f^{(\ord)}\|_{L^2(-1,1)}}{\ord!}\min_{\substack{p \text{ monic}\\ \deg p=\ord}}\|p\|_{L^2(-1,1)}.
	\end{align*}
	It is well-known that the unique minimizers are given by the monic Legendre polynomials
	\begin{equation*}
	P_\ord(y)=\frac{\ord!}{(2\ord)!}\frac{\diff^\ord}{\diff y^\ord}\left[(y^2-1)^\ord\right].
	\end{equation*}
	Since we could not find a convenient reference explaining this fact, we point out that
	the minimising property in the $L^2(-1,1)$ norm follows from the construction of the monic Legendre polynomials as the orthogonalization of powers $1,x,x^2,\ldots$ on the interval $-1\leq x \leq 1$  with respect to the Lebesgue measure. From this it follows that any monic polynomial $p$ of order $\ord$ can be written in the form $p(x)=P_\ord+c_{\ord-1}P_{\ord-1}(x)+\ldots+c_0$. Hence,
	\begin{equation*}
	\left\Vert p\right\Vert^2_{L^2(-1,1)}=\left\Vert P_\ord \right\Vert_{L^2(-1,1)}^2+\sum_{n=1}^{\ord-1}c_n^2 \left\Vert P_n \right\Vert_{L^2(-1,1)}^2.
	\end{equation*}
	Clearly, this expression is minimised by choosing $c_0=c_1=\ldots=c_{\ord-1}=0$.
	Since the  $L^2(-1,1)$ norm of the monic Legendre polynomial (see, e.g. \cite{lima2022lecturenoteslegendrepolynomials}) is
	\begin{align*}
	\left\Vert P_\ord \right\Vert_{L^2(-1,1)}^2
	&=\left(\frac{\ord!}{(2\ord-1)!!\sqrt{\ord+\frac12}}\right)^2.
	\end{align*}
	we conclude that,
	\begin{align*}
		\left\Vert f\right\Vert_{L^1(-1,1)} \leq \frac{1}{(2\ord-1)!!\sqrt{\ord+\frac12}}\left\Vert f^{(\ord)} \right\Vert_{L^2(-1,1)}.
	\end{align*}
	
	\subsection{Sharpness of the inequality in Theorem \ref{theorem}}
	
	While the proof of the sharpness of the inequality in Theorem \ref{theorem} is implicitly contained in \cite{Kalyabin2}, 	we provide a proof for the sake of completeness.
	\begin{lemma}
	For all integers $k \geq 0$, we have equality in (\ref{sob_ineq}) if and only if $f$ is equal to a constant multiple of the Landau kernel $L_k(x)$. In particular, the best constant of the equality is given by
		\begin{equation*}
			\frac{\left\Vert L_\ord\right\Vert_{L^1(-1,1)}}{\left\Vert L_\ord^{(\ord)} \right\Vert_{L^2(-1,1)}}= \frac{1}{(2\ord-1)!!\sqrt{\ord+\frac12}}.
		\end{equation*}
	\end{lemma}
	\begin{proof}
		Recall that $L_\ord(x)=(1-x^2)^\ord$. Writing $L_k(x)=(1-x)^k(1+x)^k$, we obtain 
		by repeated integration by parts that
		\begin{equation*}
			\int_{-1}^{1} |L_\ord(x)| \diff x=\frac{(\ord!)^2}{(2\ord)!}\frac{2^{2\ord+1}}{2\ord+1}
		\end{equation*}
		and, moreover, that
		\begin{equation*}
			\begin{split}
				\int_{-1}^{1}\left(L_\ord^{(\ord)}(x)\right)^2 \diff x&= (\ord!)^2\frac{2^{2\ord+1}}{2\ord+1}.
			\end{split}
		\end{equation*}
		We conclude that
		\begin{equation*}
		\frac{\left\Vert L_\ord \right\Vert_{L^1(-1,1)}}{\left\Vert L_\ord^{(\ord)} \right\Vert_{L^2(-1,1)}}=
		\frac{\ord!}{(2\ord)!}
		\sqrt{\frac{2^{2\ord+1}}{2\ord+1}}= \frac{1}{(2\ord-1)!!\sqrt{\ord+\frac12}}.
		\end{equation*}
	\end{proof}

\bibliographystyle{plain}

\end{document}